\documentclass[11pt]{article}

\usepackage[margin=1in]{geometry}
\usepackage{hyperref, xcolor}%
\definecolor{bluish-green}{HTML}{009E73}
\definecolor{reddish-purple}{HTML}{C22E99}
\definecolor{vermilion}{HTML}{C52E00}
\definecolor{nice-blue}{HTML}{0052A2}
\definecolor{nice-orange}{HTML}{E69F00}
\definecolor{marroon}{HTML}{800000}
\hypersetup{%
	breaklinks,%
	colorlinks=true,%
	urlcolor=nice-blue,%
	linkcolor=nice-blue,%
	citecolor=marroon,%
	filecolor=black, anchorcolor=black%
}
\usepackage{amsfonts, url, color}
\usepackage{amssymb,amsmath,amsthm}
\usepackage{latexsym}
\usepackage{graphicx}
\usepackage[none]{hyphenat}

\usepackage{tikz}
\usetikzlibrary{matrix,arrows,snakes,shapes,trees,positioning}

\newtheorem{theorem}{Theorem}[section]

\newtheorem{conjecture}[theorem]{Conjecture}
\newtheorem{lemma}[theorem]{Lemma}

\newtheorem{defn}[theorem]{Definition}

\newtheorem{proposition}[theorem]{Proposition}
\newtheorem{corollary}[theorem]{Corollary}

\theoremstyle{definition}

\newcounter{tenumerate}

\newcommand{\rk}{\mathrm{rk}}
\renewcommand{\epsilon}{\varepsilon}

\newcommand{\eps}{\varepsilon}

\renewcommand{\dim}{\mathsf{dim}}

\newcommand{\R}{{\mathbb R}}

\newcommand{\F}{{\mathbb F}}

\newcommand{\remove}[1]{}

\renewcommand{\leq}{\leqslant}
\renewcommand{\geq}{\geqslant}

\def\XXint#1#2#3{{\setbox0=\hbox{$#1{#2#3}{\int}$}
\vcenter{\hbox{$#2#3$}}\kern-.5\wd0}}

\renewcommand{\H}{\mathcal{H}}

\newcommand{\edc}{\mathsf{E}^{d,C}}
\newcommand{\eds}{ \mathsf{E}^{d,S}}
\newcommand{\edt}{\mathsf{E}^{d,T}}

\newcommand{\sstr}{\mathsf{sstr}}
\newcommand{\str}{\mathsf{str}}
\newcommand{\vc}{\mathsf{VC}}
\newcommand{\sign}{\mathsf{sign}}
\newcommand{\setof}[1]{#1}

\renewcommand{\int}{\mathsf{intdeg}}
\newcommand{\intdeg}{\mathsf{intdeg}}

\newcommand{\bitsz}{\{0,1\}^n}

\begin{document}
	\title{Shattered Sets and the Hilbert Function} 
	\author{
		Shay Moran\thanks{Google AI, Princeton. \url{shaymoran1@gmail.com}} 
		\and  
		Cyrus Rashtchian\thanks{University of California, San Diego. \url{crashtchian@eng.ucsd.edu}}
	}
	\date{\today}
	\maketitle

\begin{abstract} 
We study complexity measures on subsets of the boolean hypercube and exhibit connections between algebra (the Hilbert function) and combinatorics (VC theory). These connections yield results in both directions.   Our main complexity-theoretic result proves  that many linear program feasibility problems cannot be computed by polynomial-sized constant-depth circuits. Moreover, our result applies to a stronger regime in which the hyperplanes are fixed and only the directions of the inequalities are given as input to the circuit. We derive this result by proving that a rich class of extremal functions in VC theory cannot be approximated by low-degree polynomials. We also present applications of algebra to combinatorics.  We provide a new algebraic proof of the Sandwich Theorem, which is a generalization of the well-known Sauer-Perles-Shelah Lemma. 
Finally, we prove a structural result about downward-closed sets, related to the Chv\'{a}tal  conjecture in extremal combinatorics.  
\end{abstract}

\section{Introduction} 
Understanding  the properties and structure of subsets of the boolean hypercube is a central theme in theoretical computer science and combinatorics.   When studying a family of mathematical objects, endowing the objects with algebraic structure often sheds new light on interesting properties.  This phenomena appears classically in areas such as algebraic topology and algebraic geometry.  It also provides much utility when studying the boolean hypercube.  
Let $C\subseteq \{0,1\}^n$ and $\F$ be a field. Consider the linear space of functions from $C$ to $\F$, that is,  $\F^{C}$.  This is clearly a $|C|$-dimensional vector space over $\F$. Every function in this space can be represented as a multilinear polynomial with degree at most $n$. Interestingly, for certain sets $C$, smaller degree actually suffices. For example, when $C$ is the standard basis, denoted $C=\{\vec e_1,\ldots,\vec e_m\}$, then any function $f:C\rightarrow\F$ can be expressed as the linear function $f(\vec e_1)x_1+\ldots +f(\vec e_m)x_m$.   In other words, degree one suffices. 

The Hilbert function, denoted  $h_d(C,\F)$, is the dimension of the space of functions $\{f:C\to \F\}$ that have representations as polynomials with degree at most $d$.  
This classical algebraic object will be useful in our study of how the structure of $C$ affects the function space.
In complexity theory, Smolensky~\cite{smolensky} has used the Hilbert function to unify 
polynomial approximation lower bounds relating to bounded-depth circuits.   

We establish new connections between the Hilbert function and VC theory. VC theory lies in the intersection of statistics and combinatorics and has connections to many areas of computer science and mathematics, such as discrete geometry, machine learning, and complexity theory. 

The main technical contribution of this paper is an upper and a lower bound on $h_d(C,\F)$ in terms of 
basic concepts in VC theory, such as shattering, strong shattering, and down-shifts.  This connection between Algebra and Combinatorics allows us to derive results in both directions.
Our main complexity theoretical application is that determining feasibility of a large family of linear programs is hard for the class of bounded-depth circuits.  
More specifically, let $h_1,\ldots,h_m$ be affine functions. Each sign vector $s$ in $\{\pm\}^m$ defines the following feasibility problem: does there exist $x\in \R^d$ such that $h_i(x) > 0$ when $s_i=+$, and $h_i(x) < 0$ when $s_i=-$, for all $i \in [m]$?  This defines a boolean function that takes an input $s$ and outputs one if and only if the problem is feasible.  We prove that if $m=2d+1$, and the affine functions $h_i$ are in general position, then computing this function requires a constant-depth circuit of exponential size.
We do so by proving that this function cannot be approximated by low-degree polynomials, over any field.  The circuit lower bound then follows from the polynomial approximation technique introduced by Razborov~\cite{razborov} and Smolensky~\cite{smolensky87}.  

As a combinatorial application of our bounds on the Hilbert function, we  provide a short algebraic proof of the Sandwich Theorem.  This theorem comes from VC theory and is a well-studied generalization of the Sauer-Shelah-Perles Lemma~\cite{Law, pajor,BR89, br, dress,ShatNews, Dress2,moran_thesis,KozmaM13,MeszarosR14,MoranW15}.   

Facts we prove about the function space $\F^C$ also lead to a new result about downward-closed sets.  A family~$D$ of subsets is {\em downward-closed} if $b \subseteq a$ and $a \in D$ implies $b \in D$.   A theorem of Berge~\cite{berge} implies that for any downward-closed set $D$  there exists a bijection $\pi: D \to D$ such that $a \cap \pi(a) = \emptyset$ for all $a \in D$. We generalize his result to arbitrary, prescribed intersections.  Let $\phi : D \to D$ have the property $\phi(a) \subseteq a$ for all $a \in D$.  We show that there always exists a bijection $\pi: D \to D$ such that $a \cap \pi(a) = \phi(a)$.   Note that choosing $\phi(a) = \emptyset$ for all $a$ implies Berge's result.

Our algebra-combinatorics connection fits within the framework of the polynomial method. This method has been successful in providing elegant proofs of fundamental results in many areas, such as circuit complexity~\cite{smolensky87, abfr, razborov, beigel}, discrete geometry~\cite{guthkatz, dvir, sharir, tao}, extremal combinatorics~\cite{alon-null, jukna-extremal, babai-frankl}, and more.   

The rest of the paper is organized as follows.  We state our main theorems in Section~\ref{section:results}.  In Section~\ref{section:hilbert}, we prove our bounds on the Hilbert function.   In Section~\ref{section:lp}, we use our Hilbert function bounds to prove that linear program feasibility is hard for bounded-depth circuits.  Finally, in Section~\ref{section:downward-closed}, we prove results about downward-closed sets.   We now review preliminaries.

\subsection{Preliminaries}  
We begin with algebraic preliminaries.   Let $C \subseteq \{0,1\}^n$ and $\F$ be a field.
Every $f : C \to \F$ can be expressed as a multilinear polynomial over variables $x_1,\ldots,x_n$ with coefficients in~$\F$.   

\begin{defn} For $d \in [n]$ the {\em Hilbert function} $h_d(C,\F)$ is the dimension of the space of functions $f: C \to \F$ that can be represented as polynomials with degree at most $d$. 
\end{defn}
\noindent
Notice that $h_d(C,\F) \leq \min\{\sum_{j=0}^d {n \choose j}, |C|\}$.  
A basic fact about the Hilbert function is that 
\[ 1 = h_0(C,\F) \leq h_1(C,\F) \leq \ldots \leq h_n(C,\F) =  |C|.\]
The final equality holds because all $f:C\to \F$ have representations with degree at most $n$.

It is natural to wonder when the Hilbert function attains its maximum and how the structure of $C$ influences the Hilbert function.  
We introduce the following as a complexity measure on sets. 
\begin{defn}
The {\bf interpolation degree of $C$} denoted $\int(C)$ is the minimum $d$ such that any $f:C \to \F$ can be expressed as a multilinear polynomial with degree at most $d$.  In other words, 
\[ \int(C) = \min \{ d \in [n]  :  h_d(C,\F) = |C|\}. \]
\end{defn}
\noindent
Intuitively, a smaller interpolation degree implies a less complex subset of the boolean hypercube. 

We move on to combinatorial preliminaries.  Our bounds on the Hilbert function use basic concepts in VC theory.  We define these concepts now.  

\begin{defn} A subset $I\subseteq [n]$ is {\bf shattered} by $C \subseteq \{0,1\}^n$ if for every pattern $s:I\rightarrow\{0,1\}$ there is $c\in C$ that realizes $s$. In other words, the restriction of $c$ to $I$ equals $s$. 
A subset $I\subseteq [n]$ is {\bf strongly shattered} by $C$ if $C$ contains all elements of some subcube on $I$. In other words, there exists a pattern $\bar s:([n] \setminus I)\rightarrow\{0,1\}$ such that all extensions of $\bar s$ to a vector
in $\{0,1\}^{n}$ are in $C$. 
\end{defn}
\noindent
These definitions lead to natural families of sets, which will be important to our work. 

\begin{defn}
The {\bf shattered sets} with respect to $C$ are 
\[ \str(C) = \{I\subseteq [n]:I\mbox{ is shattered by }C\}.\]
The {\bf strongly shattered sets} with respect to $C$ are  
\[ \sstr(C) = \{I\subseteq [n]:I\mbox{ is strongly shattered by }C\}.\]
\end{defn}

\begin{defn}
The {\bf VC dimension} of $C$ is defined as $ \vc(C) = \max \{|I| : I \in \str(C) \}$.
\end{defn}
\noindent
Note that $\sstr(C)\subseteq\str(C)$ and that both of these families are downward-closed. 

We also lower bound the Hilbert function using down-shifts, a  standard tool in extremal combinatorics.   
Let $C\subseteq\{0,1\}^n$ and let $i\in [n]$. We denote as $S_i$ the down-shift operator
on the $i$'th coordinate.  Obtain the set 
$S_i(C)\subseteq\{0,1\}^n$ from $C$ as follows.  Replace every $c\in C$ such that both (i) $c_i=1$, and (ii) the $i$'th neighbor\footnote{Vectors $u,v\in\{0,1\}^n$ are {\em $i$'th-neighbors} if they differ in coordinate $i$ and are the same elsewhere.} of $c$ is not in $C$ with the $i$'th neighbor of $c$.    Authors have referred to this operation as ``compression'', ``switching'', and ``polarization'. Previous works that uses down-shifts include  \cite{Kleitman1966,Enflo70,BollobasL96,GoldreichGLRS00,GreenT09,moran_thesis,Odonnelbook}.


The most important property of down-shifts for our results is that they transform an arbitrary subset of $\{0,1\}^n$ into a downward-closed set, without changing cardinality. Specifically, if \[D=S_{n}(S_{n-1}(\ldots S_1(C)))\] is the result of sequentially applying $S_i$ on $C$ for each $i$, then $\setof{D}$ is downward-closed. 
(It is convenient in this context to think of $D$ as a family of subsets of $[n]$ rather than a set of boolean vectors via the natural correspondence between boolean vectors and sets.)

We move on to explaining our results in more formality and detail.


\section{Our Results}  \label{section:results}
We start with the result about linear program feasibility.  We then state the bounds on the Hilbert function in terms of shattered sets and down-shifts.  We show this leads to bounded-depth circuit lower bounds.  Finally, we state two combinatorial applications. 
\subsection{Linear Program Feasibility}\label{sec:result_lp}

We formalize and prove a strong version of the statement ``linear programming feasibility can not be decided by polynomial-sized, constant-depth circuits.''  Clearly, linear programming being P-complete~\cite{dobkin} implies a version of this statement for specific linear programs representing functions previously known not to have efficient bounded-depth circuits.  
We prove a stronger version stating that any linear feasibility problem, in which the number of constraints is roughly twice the number of variables and the constraints are non-degenerate, cannot be decided by an efficient bounded-depth circuit.   
For a set of hyperplanes $\mathcal{H}$ in $\R^k$ we will define a boolean function $f_\mathcal{H}$.  It takes orientations as inputs and outputs one if and only if a certain polytope is nonempty.  In particular, we establish hardness of this problem even when the hyperplanes are fixed in advance and only the orientations are given as input.   


We express linear program feasibility as a boolean function as follows.
Specify an arrangement of $m$ hyperplanes $\mathcal{H}=\{h_1,\ldots,h_m\}$ with normal vectors $\vec n_i$ and translation scalars $b_i$ as 
$$h_i=\{\vec{x}:\langle\vec{n_i},\vec{x}\rangle=b_i\}.$$
\noindent
A sign pattern $s\in\{- 1,1\}^m$ encodes the following linear programming feasibility problem:   
\begin{center}
{\em Does there exist $x\in\R^k$ satisfying $\sign(\langle x,n_i\rangle - b_i) = s_i$ for all $i \in [m]$?}
\end{center}

This corresponds to checking the feasibility of a linear program with $m$ constraints and $k$ variables.  Define $f_\mathcal{H}:\{- 1,1\}^m\rightarrow\{0,1\}$ as the boolean function such that $f_\mathcal{H}(s)=1$ if and only if the linear program encoded by $s$ is feasible.

As an example, consider the following arrangement in $\R^2$.  The three hyperplanes
\begin{align*}
&h_1 = \{(x_1,x_2):5x_1+3x_2=3\}, h_2=\{(x_1,x_2):8x_1-x_2=8\},
h_3=\{(x_1,x_2):4x_1-5x_2=0\}
\end{align*}
form an arrangement of three lines in the plane. 
The vector $s=(+1,-1,+1)$ encodes the system
\begin{align*}
&5x_1+3x_2>3 \tag{$s(1)=+1$}\\
&8x_1-x_2<8\tag{$s(2)=-1$}\\
&4x_1+5x_2>0\tag{$s(3)=+1$}
\end{align*}
\noindent
In the  example, the system encoded by $(+1,-1,+1)$ is not satisfiable (see Figure~\ref{fig:arrangement}).

\begin{figure}[t]
\begin{center} 
\begin{tikzpicture}[scale=1.5]

\draw[-,thick] (0,3) -- (4,0.6);
\node at  (-1,3.2)
{\footnotesize $h_1:3x_1+5x_2=3$};

\draw[-,thick] (1,0) -- (1.5,3.7);
\node at  (1.5,3.8)
{\footnotesize $h_2:8x_1-x_2=8$};

\draw[-,thick] (-0.2,0) -- (4,3);
\node at  (4.6,3.2)
{\footnotesize $h_3:4x_1-5x_2=0$};

\node at  (3.5,2)
{\footnotesize $+++$};
\node at  (2.5,3)
{\footnotesize $++-$};
\node at  (1.58,1.65)
{\footnotesize $-+-$};
\node at  (2.25,0.75)
{\footnotesize $-++$};
\node at  (0.65,0.25)
{\footnotesize $--+$};
\node at  (0.6,1.5)
{\footnotesize $---$};
\node at  (0.75,3)
{\footnotesize $+--$};
\end{tikzpicture}
\end{center}
\caption{Three lines divide $\mathbb{R}^2$ into seven regions, each labeled by a feasible sign pattern.} \label{fig:arrangement}
\end{figure}
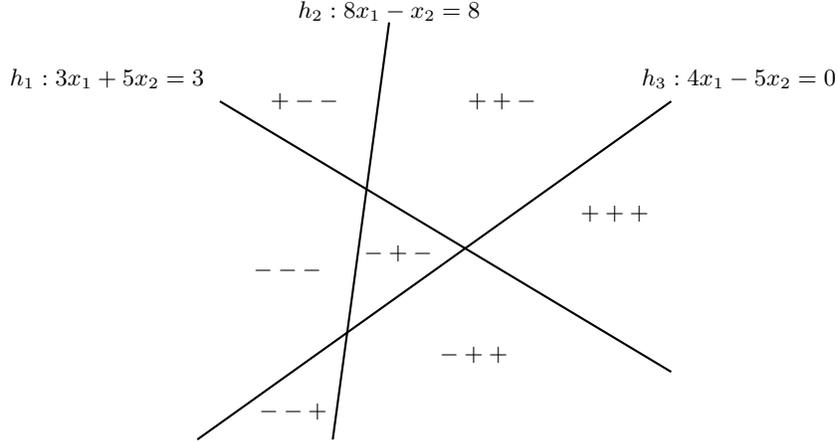


We prove the following theorem.
\begin{theorem}\label{thm:hyperplanes}
Let $\mathcal{H}$ be an arrangement of $2k+1$ hyperplanes in $\R^k$ that are in general position.  Any $AC^0[p]$ circuit, for a prime $p$, with depth $d$ computing $f_\H$ requires {$\exp(\Omega(k^{1/2d}))$} gates. 
\end{theorem}
\noindent
We prove Theorem~\ref{thm:hyperplanes} in Section~\ref{sec:hyperplanes}, using
the framework of Razborov~\cite{razborov} and Smolensky~\cite{smolensky87}. 

\paragraph{Explicit Arrangements.}
The space of oriented hyperplanes is a rich and well-studied object.  The book~\cite{stanley} provides many facts and examples.  The paper~\cite{alon-geom} and references therein give bounds on how many different boolean functions can be represented as $f_\mathcal{H}$ for some hyperplane arrangement $\mathcal{H}$.

General position hyperplane arrangements come from any $2k+1$ vectors in $\R^{k+1}$ such that every $k+1$ of them are linearly independent.   For a vector $v \in \R^{k+1}$ the hyperplane has normal $(v_1,\ldots,v_k)$ and translation~$v_{k+1}$. 
Explicit families of $m$ vectors in $\R^{d}$ such that every $d$ of them are independent are known for any $m,d$. For example, take the rows of an $m\times d$ Cauchy or Vandermonde matrix. 

\subsection{Hilbert Function Bounds}
The core of our technical contribution is the following theorem.   
\begin{theorem}\label{rank}\label{main} Any $C \subseteq \{0,1\}^n$ and any $d \in [n]$ satisfy the  relationships 
\[ |\{ I \in \sstr(C)  \  :  \ |I| \leq d\}| \ \leq\  h_d(C,\F)\  \leq \ |\{ I \in \str(C)  \  :  \ |I| \leq d\}| \] and  
\[ \max\{|I|  :  I\in \sstr(C)\} \ \leq\ \int(C)\ \leq \ \max\{|I|  :  I\in \str(C)\}.\] 
\end{theorem}
\noindent 
We prove Theorem~\ref{main} in Section~\ref{sec:proofmain}

{We note that the upper} bounds on interpolation degree in terms of VC dimension appear implicitly in the work of Frankl and Pach~\cite{babai-frankl} and explicitly in Gurvits~\cite{gurvits}, and in~\cite{DBLP:journals/cc/Smolensky97b}, {who use similar arguments to derive them.}

We strengthen the upper bound on the Hilbert function in Theorem~\ref{main} using down-shifts. 
\begin{theorem}\label{thm:shifthilbert}
Let $C\subseteq\{0,1\}^n$ and let $D=S_{n}(S_{n-1}(\ldots S_{1}(C)))$.
Then
\[ \int(C) \ \leq\   \max\{|I|  :  I\in D\}.\] 
\end{theorem}
\noindent
This Theorem is proven In Section~\ref{sec:shift}.  
We also discuss how Theorem~\ref{thm:shifthilbert} implies the upper bound in Theorem~\ref{main}.
\noindent

\subsection{Low-Degree Polynomial Approximations}
Classic results in bounded-depth circuit complexity utilize the fact that small circuits have low-degree approximating polynomials to reduce the task of proving circuit lower bounds to showing that a boolean function has no low-degree approximation~\cite {razborov, smolensky87, abfr}.   Smolensky shows in~\cite{smolensky} how to express all known degree lower bounds in terms of the Hilbert function.   For a boolean function $f$ consider the set $S = f^{-1}(1)$ as a subset of the boolean cube.  Smolensky shows that if $h_d(S,\F)$ is large, then $f$ is hard to approximate by low-degree polynomials. 
\begin{theorem}[\cite{smolensky}]\label{smolensky} Consider $f:\bitsz \to \{0,1\}$ and $p:\{0,1\}^n \to\F$.  Define $S = f^{-1}(1)$ and fix $d = \lfloor (n-\deg_\F(p)-1)/2 \rfloor$.  Then,  
\[ \Pr_x[p(x) \neq f(x)]  \geq \frac{2\cdot h_d(S,\F) - |S|}{2^{n}},\] where $x$ is uniform over $\{0,1\}^n$.
\end{theorem}
\noindent
We provide a proof of Smolensky's result in Appendix~\ref{app:smolensky} for completeness.  Theorem~\ref{main} implies the following corollary in terms of strongly shattered sets. 

\begin{corollary}\label{sstr-bound} Assume $n$ is odd. Consider $f:\bitsz \to \{0,1\}$.  If $|f^{-1}(1)|~=~2^{n-1}$ and $\sstr(f^{-1}(1))=\{Y\subseteq[n]:\lvert Y\rvert\leq\frac{n-1}{2}\}$, then for any  polynomial $p \in \F[x_1,\ldots,x_{n}]$ we have 
\[ \Pr_x[p(x) \neq f(x)] \geq \frac{1}{2} - \frac{10\deg_\F(p)}{\sqrt{n}}, \]
where $x$ is uniform over $\{0,1\}^n$. 
\end{corollary}
\begin{proof}  Since $\sstr(F)=\{Y\subseteq[n]:\lvert Y\rvert\leq\frac{n-1}{2}\}$, we have that \[ |\{ I \in \sstr(C)  \  :  \ |I| \leq d\}|  = \sum_{j=0}^d {n \choose j}\] for all $d = 0,1, \ldots, (n-1)/2$.   Theorem~\ref{main} implies that $h_d(f^{-1}(1),\F) = \sum_{j=0}^d {n \choose j}$  as well.  Plugging these into Theorem~\ref{smolensky}, along with the assumption $|f^{-1}(1)| = 2^{n-1}$, gives the corollary. \end{proof}

Bernasconi and Egidi~\cite{bernasconi} thoroughly characterize the Hilbert function for symmetric sets and prove that any nearly-balanced, symmetric boolean function is hard to approximate.  They leave as an open question deriving bounds for non-symmetric sets.   Our connection to VC  theory leads to new families of functions satisfying the conditions of Corollary~\ref{sstr-bound}.  Many of these functions, such as the linear programming feasibility functions from Section~\ref{sec:result_lp}, are non-monotone and non-symmetric. As a final remark, recent work shows that Smolensky's lower bound (and thus our result) extends  to nonclassical polynomials~\cite{nonclassical}.

\subsection{The Sandwich Theorem}
The following relationship which is a generalization of the Sauer-Perles-Shelah Lemma was discovered several times and independently~\cite{br, pajor, dress,ShatNews}. 
\begin{theorem}[Sandwich Theorem] \label{thm:sandwich} For any $C \subseteq \{0,1\}^n$ we have 
$ |\sstr(C)| \leq |C| \leq |\str(C)|.$
\end{theorem}
\noindent
Since $|\str(C)| \leq \sum_{i=0}^{\vc(C)} {n \choose i}$, this implies the usual formulation of the Sauer-Perles-Shelah  Lemma. 

Theorem~\ref{main} yields a new algebraic proof of the Sandwich Theorem.
Indeed, this follows from examining the case of $d=n$ and observing that $h_n(C,\F)=|C|$.

The Sandwich Theorem is tight in the sense that there are sets that achieve equality in both of its inequalities\footnote{In fact, it is well known (see for example~\cite{moran_thesis}) that any set achieving equality in one of the inequalities, also achieves equality in the other.}.
These sets are calles {\em shattering extremal sets}. For example, downward-closed sets are shattering extremal.
Shattering extremal sets have been rediscovered and studied in different contexts~\cite{Law,BR89,br,dress,Dress2,moran_thesis,KozmaM13,MeszarosR14,MoranW15}. 
In our context, Corollary~\ref{sstr-bound} says that shattering extremal sets $S$ of size $|S| = 2^{n-1}$ and VC dimension $\frac{n-1}{2}$ correspond to boolean functions that cannot be approximated by low-degree polynomials.  

\subsection{Downward-closed Sets and Chv\'{a}tal's Conjecture} 
Downward-closed sets have a well-studied, rich combinatorial structure.
A theorem of Berge~\cite{berge} implies the following fact.  For any downward-closed set $D$, there is a bijection $\pi: D \to D$ such that $a \cap \phi(a) = \emptyset$, for all $a \in D$.  We refer to such a bijection as a {\em pseudo-complement}. We prove the following generalization of the existence of a pseudo-complement.

\begin{theorem}~\label{pseudo} 
Let $D$ be any downward-closed set. Fix any mapping $\phi : D\to D$ with the property that $\phi(a) \subseteq a$ for all $a \in D$.   Then there exists a bijection $\pi : D\to D$ satisfying the condition that $a \cap \pi(a)= \phi(a)$ for all $a\in D$. 
\end{theorem}
\noindent 
Note that choosing $\phi(a) = \emptyset$ for all $a$ implies the existence of a  pseudo-complement.

In topology, downward-closed sets correspond to simplicial complexes.  We think of the $\phi$ as prescribing intersections. For simplicial complexes, this corresponds to prescribing that complexes intersect in certain faces.  We prove Theorem~\ref{pseudo} in Section~\ref{sec:pseudo}.  Our proof proceeds by proving that a certain matrix is invertible and thus its determinant is non-zero. Having non-zero determinant implies that the matrix contains a permutation matrix that yield the desired bijection.

We next discuss the result by Berge for the existence of pseudo-complements and its connections with Chv\'{a}tal's conjecture in extremal combinatorics~\cite{chvatal}.
Berge's result about pseudo-complements follows from the following stronger theorem that he proved.
\begin{theorem}[\cite{berge}]\label{berge}  If $D$ is a downward-closed set, then either $D$ or $D \setminus \emptyset$ can be partitioned into pairs of disjoint sets.  
\end{theorem}
\noindent
We need two definitions to explain Berge's motivation.   A family~$B$ of subsets of $[n]$ is called a {\em star} if there is an element $x\in [n]$ such that $x \in b$ for all $b \in B$.  It is called an {\em intersecting family} if every pair of sets in $B$ intersects.  Chv\'{a}tal's  conjecture is the following.
\begin{conjecture}[Chv\'{a}tal's conjecture]  If $D$ is a downward closed set, then the cardinality of the largest  star in $D$ is equal to the cardinality of the largest intersecting family in $D$.
\end{conjecture}
\noindent
This conjecture remains open, aside from partial results, such as a corollary of Berge's theorem.  
\begin{corollary} 
In a downward-closed set $D$, any intersecting family has cardinality at most $|D|/2$.
\end{corollary}

We contrast Berge's theorem and our Theorem~\ref{pseudo}.  Berge's pair decomposition induces a permutation~$\pi$ such that $\pi(\pi(a)) = a$, whereas a permutation decomposes $D$ into disjoint cycles with unspecified lengths.  Many people have observed that the above corollary only needs the pseudo-complement result, instead of the stronger statement in Berge's theorem~\cite{anderson}.  Indeed, consider each disjoint cycle in the guaranteed permutation, and note that at most half of the sets in the cycle may mutually intersect.   Therefore, our Theorem~\ref{pseudo} implies the above corollary.

\section{The Hilbert Function for Subsets of the Boolean Cube} \label{section:hilbert}
We prove upper and lower bounds on the Hilbert function.  First, we prove the bounds in Theorem~\ref{main} involving the shattered and the strongly shattered sets. Then, we prove the bound in Theorem~\ref{thm:shifthilbert} using shifting. Finally we consider an example of  applying these bounds to analyze the Hilbert function of the parity function.  
\subsection{Bounding the Hilbert Function Using Shattered Sets}\label{sec:proofmain}

The high-level idea of the proof of Theorem~\ref{rank} is to define a vector space $V$ with $\dim(V)=|C|$ and prove that $|\sstr(C)|\leq \dim(V) \leq |\str(C)|$.  We sandwich the dimension $\dim(V)$ by finding a linearly independent set of size $|\sstr(C)|$ and a spanning set of size $|\str(C)|$. 

We analyze the $|C|$-dimensional vector space $\{ f: C\rightarrow\mathbb F\}.$
Evaluation on $C$ induces a natural mapping from $P \in \F[x_1,\ldots,x_n]$ to the restriction $P|_C \in \{f: C \to \F\}$.  
 The following lemma provides the desired sets of spanning monomials and linearly independent monomials, respectively. 
\begin{lemma}\ For all fields $\F$ and sets $C \subseteq \{0,1\}^n$ the following two facts hold. 
\begin{enumerate}
\item The monomials $\prod_{i\in Y} x_i$ for $Y\in\str(C)$ span $\{f:C\rightarrow \F\}$.
\item The monomials  $\prod_{i\in Y} x_i$ for $Y\in\sstr(C)$  are linearly independent in $\{f:C\rightarrow \F\}$.
\end{enumerate}
\end{lemma}
\begin{proof}

For $Y\subseteq [n]$, let $x_Y$ denote the monomial $x_Y=\prod_{i\in Y}x_i$.  For the first item, we express every $f:C\rightarrow \F$ as a linear combination of
monomials $(x_Y)|_C$ where $Y\in\str(C)$.  
It suffices to express the monomials $(x_Y)|_C$ for all $Y \subseteq [n]$.   We prove this by induction.  For the base case, if  $Y \in\str(C)$, we are done.  
Otherwise, $Y$ is not shattered by $C$ and
there exists $s\in\{0,1\}^Y$ such that for all $c\in C$, we have 
$c|_Y \neq s.$
Consider 
$$P=\prod_{i\in Y}\left(x_i-\left(1-s_i\right)\right).$$
Note that $P(c)=0$ for all $c\in C$ and hence
$P|_C = 0|_C.$
Specifically, by expanding the product $\prod_{i\in Y}\left(x_i-\left(1-s_i\right)\right)$ we see
$$0|_C = P = (x_Y)|_C + (Q)|_C,$$
where the degree of $Q$ is smaller than $|Y|$.   By induction, we can write $Q$ as a combination of $x_I$ for $I \in \str(C)$.
Since $(x_Y)|_C = (-Q)|_C$ we get that $x_Y$ is in this span as well.

We now prove the second item.
Consider a linear combination
$$P=\sum_{Y\in\sstr(C)}{\alpha_Yx_Y}$$
such that not all $\alpha_Y$ are zero. 
We want to show that there is $c\in C$
such that $P(c)\neq 0$.
Let $Z\in\sstr(C)$ be a maximal set such that $\alpha_Z\neq0$.
Since $Z$ is strongly shattered by $C$, there is some $\bar s:([n]\setminus Z)\rightarrow\{0,1\}$
such that all extensions of it in $\{0,1\}^n$ are in $C$. Let $Q(x_i)_{i\in Z}$ be the polynomial 
obtained by plugging in the values of $c$ in the variables of $([n]\setminus Z)$.
By maximality of $Z$ it follows that the coefficient of $x_Z$ in $Q$ is $\alpha_Z\neq 0$, and so $Q$ is not the $0$ polynomial.
Therefore there is $s\in\{0,1\}^Z$ such such that $Q(s)\neq 0$.
Pick $c\in C$ such that
$$c_i = \begin{cases}
s_i  &i\in Z,\\
\bar s_i &i\in ([n]\setminus Z).
\end{cases}$$
It follows that $P(c) = Q(s)\neq 0$, which finishes the proof.
\end{proof}
We use this lemma to prove our results bounding the Hilbert function and interpolation degree. 
\begin{proof}[Proof of Theorem~\ref{rank}]   For the upper bound on $h_d(C,\F)$, the above proof shows how to express all monomials of degree $d$ using monomials of equal or smaller degree.   For the lower bound on $h_d(C,\F)$,  linear independence still holds after restricting set size. 

The upper bound on $\intdeg(C)$ is immediate. For the lower bound on $\intdeg(C)$,  since $\sstr(C)$ is downward-closed, the linear independence of the monomials in $\sstr(C)$ implies any maximal degree monomial in $\{(x_Y)|_C : Y\in\sstr(C)\}$ cannot be expressed solely by  lower degree monomials.  
\end{proof}

\subsection{Down-shifts, Downward-closed Bases, and the Hilbert Function}\label{sec:shift}

We prove here Theorem~\ref{thm:shifthilbert}, which is a direct corollary of the following theorem.
\begin{theorem}\label{thm:shifting}
Let $C\subseteq\{0,1\}^n$ and let $D=S_{n}(S_{n-1}(\ldots S_{1}(C)))$.
Then the set of monomials $\{\prod_{i\in I} x_i  :  I\in \setof{D}\}$
is a basis for the vector space of functions $\{f:C\rightarrow \F\}$.
\end{theorem}
\noindent A theorem, equivalent in content, but expressed with respect to  Gr\"{o}bner bases, is proved in~\cite{Meszaros_thesis}. For completeness we include an elementary proof in Appendix~\ref{app:shifting}.

The upper bound given in Theorem~\ref{thm:shifthilbert} subsumes the upper bound in Theorem~\ref{main}. This is a direct corollary of the following simple lemma.
\begin{lemma}[Theorem 6.4 in \cite{moran_thesis}]
Let $C\subseteq\{0,1\}^n$ and let $D=S_{n}(S_{n-1}(\ldots S_1(C)))$.
We have that
$\setof{D}\subseteq \str(C)$, where we associate $\{0,1\}^n$ with subsets of $[n]$ in the natural way.
\end{lemma}
\begin{proof}
Since $\setof{D}$ is downward-closed, it follows that it is shattering extremal and therefore $\sstr(D)=\setof{D}$.
So, it is enough to show that $\sstr(C)\subseteq\sstr(D)$. 
To this end, it suffices to show that for every class $C'$, $\sstr(C')\subseteq\sstr(S_i(C'))$.
Let $I\in\sstr(C')$. Therefore $C'$ contains a subcube $B$ in coordinates~$I$.
During the down-shift, $B$ is either shifted or stays in place, but in any case also $S_i(C')$ contains a subcube in coordinates $I$
and therefore $Y\in\sstr(S_i(C'))$.
\end{proof}

\paragraph{An Example.}

A simple example which demonstrates an application of Theorem~\ref{thm:shifthilbert}
is the set
\[C=\Bigl\{100000,010001, 001010,000111\Bigr\}.\]
Note that the last two coordinates are shattered and hence the upper bound on the interpolation degree
given by Theorem~\ref{main} is $2$. However,
\[S_5(S_4(S_3(S_2(S_1(C))))) =  \Bigl\{100000,010000, 001000,000100\Bigr\},\]
and hence the upper bound implied by Theorem~\ref{thm:shifthilbert} is $1$, which is better.

\section{Linear Programming is not in $AC^0[p]$}\label{sec:hyperplanes}\label{section:lp}



We now prove Theorem~\ref{thm:hyperplanes}.    Using the Razborov-Smolensky framework, it suffices to prove the following theorem, showing $f_\mathcal{H}$ cannot be approximated by a low-degree polynomial over any field.\footnote{We state the following theorem for $\{-1,1\}$ inputs to $f_{\mathcal{H}}$.  This only makes sense for fields containing these elements.  When $\F = \F_2$ simply replace $\{-1,1\}$ with $\{0,1\}$ in the definition of $f_\mathcal{H}$.}

\begin{theorem}\label{thm:hyperplanes-polynomial}
Let $\mathcal{H}$ be an arrangement of $2k+1$ hyperplanes in $\R^k$ that are in general position. For any any polynomial $p \in \F[x_1,\ldots,x_{2k+1}]$ we have 
\[ \Pr_s[p(s) \neq f_\mathcal{H}(s)] \geq \frac{1}{2} - \frac{10\deg_\F(p)}{\sqrt{2k+1}}, \]
where $s$ is uniform over $\{-1,1\}^{2k+1}$. 
\end{theorem}
\noindent
The proof of Theorem~\ref{thm:hyperplanes-polynomial} proceeds via a reduction to Corollary~\ref{sstr-bound}. 
Let
$$S_\mathcal{H} = \{s\in\{-1,1\}^n:~f_\mathcal{H}(s)=1\}.$$
To apply Corollary~\ref{sstr-bound} on $f_\mathcal{H}$ we will show  
$\lvert S_\mathcal{H}\rvert = 2^{2k}$ and $\sstr(S_\mathcal{H})=\{Y\subseteq [2k+1] : \lvert Y\rvert\leq k\}$.  
We establish this by the following proposition.  The facts we need about hyperplane arrangements follow from standard arguments~\cite{gartner, stanley}.  For intuition about the following proposition, see Figure~\ref{fig2} for a pictorial proof  in $\R^2$.

\begin{figure}[t]
\begin{center}
\includegraphics[scale=.33]{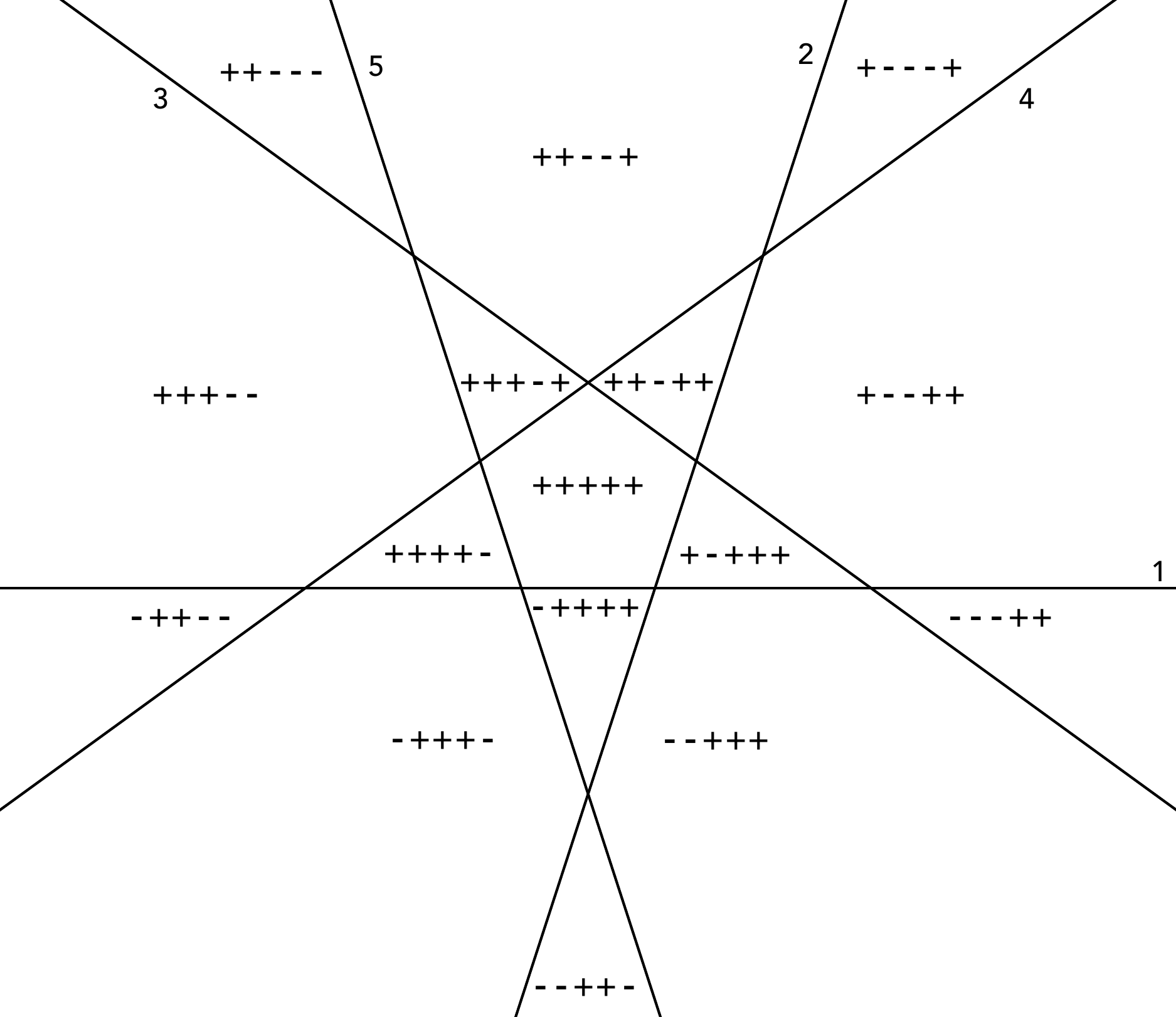}
\caption{Five hyperplanes divide $\mathbb{R}^2$ into 16 cells.  Cell labels in $\{-,+\}^5$ correspond to oriented hyperplane feasibility.   Notice that every two coordinates are strongly shattered, but no three coordinates are shattered. This provides a proof-by-picture of Proposition~\ref{thm:hyp-extremal} for $m=5$ and $d=2$.} \label{fig2}
\end{center}
\end{figure}

\begin{proposition}\label{thm:hyp-extremal}
For any $m$ hyperplanes $\mathcal{H}$ in $\R^d$ in general position 
$$\sstr(S_\mathcal{H}) = \str(S_\mathcal{H}) = \{Y\subseteq [m] : |Y|\leq d\}.$$
\end{proposition}
\begin{proof}
The two lemmas in Appendix~\ref{app:hyplerplanes} characterize the shattered and strongly shattered
sets of $S_\mathcal{H}$ when $\mathcal{H}$ is in general position. 
The first lemma shows  $\str(S_\mathcal{H})\subseteq\{Y\subseteq [m] : |Y|\leq d\}$. The second lemma shows $\{Y\subseteq [m] : |Y|\leq d\}\subseteq\sstr(S_\mathcal{H})$.
Since $\sstr(S_\mathcal{H})\subseteq\str(S_\mathcal{H})$
these two lemmas combine to finish the proof.
\end{proof}

Proposition~\ref{thm:hyp-extremal} implies Theorem~\ref{thm:hyperplanes}. 
The equality $\sstr(S_\mathcal{H}) = \str(S_\mathcal{H})$ along with the Sandwich Theorem implies that $|S_\mathcal{H}|=|\sstr(S_\mathcal{H})|$.  Let $k$ be the ambient dimension in Theorem~\ref{thm:hyperplanes}.   The above proposition for  $m=2k+1$ and $d=k$ gives $\lvert S_\mathcal{H}\rvert = 2^{2k}$ and also $\sstr(S_\mathcal{H})=\{Y\subseteq [2k+1] : \lvert Y\rvert\leq k\}$. Thus $f_\mathcal{H}$ satisfies the premises of Corollary~\ref{sstr-bound}, and Theorem~\ref{thm:hyperplanes} follows.
%



\section{Downward-closed Sets and Prescribed Intersections}\label{sec:pseudo}\label{section:downward-closed}
We prove Theorem~\ref{pseudo}.   Let $D\subseteq \{0,1\}^n$ be a downward-closed set. Fix $\phi : D\to D$ with the property that $\phi(a) \subseteq a$ for all $a \in D$.    We will show that there exists a bijection $\pi : D\to D$ satisfying the condition that $a \cap \pi(a) = \phi(a)$ for all $a\in D$.     We first prove two lemmas about the function space $\{f:D \to GF(2)\}$ and then use these to prove the existence of $\pi$.   The first lemma holds for all subsets of the boolean cube.\newpage

\begin{lemma}\label{basis1} Let $C \subseteq \{0,1\}^n$. The monomials $\displaystyle \prod_{i \in a} x_i$ for $a \in C$ form a basis for $\{f:C~\to~GF(2)\}$.
\end{lemma}
\begin{proof} We proceed using induction on $|C|$.  When $C = \{a\}$ for $a \in \{0,1\}^n$ the function space has dimension one and $\prod_{i \in a} x_i = 1$.  For general $C$, let $z \in C$ denote a maximal Hamming weight element in $C$.   Notice that $\prod_{i \in z} x_i$ is an indicator function  in $\{f:C \to GF(2)\}$ for the input $z$.  By the inductive hypothesis on $(C \setminus \{z\})$, we know the set of monomials $\prod_{i \in a} x_i$ for $a \in (C \setminus \{z\})$ form a basis for $\{f:(C \setminus \{z\})\to GF(2)\}$.  Since $\prod_{i \in z} x_i$ is an indicator function, we may add it to the basis for $\{f:(C \setminus \{z\})\to GF(2)\}$ and achieve a basis for $\{f:C \to GF(2)\}$.
 \end{proof}
We remark that if $C$ is downward-closed, then it is shattering extremal, and the above lemma is a corollary of the Sandwich theorem.  
We prove the following stronger claim as well.  
 
 \begin{lemma}\label{basis2} Let $D \subseteq \{0,1\}^n$ be a downward-closed set.  Fix any mapping $\phi : D\to D$ with the property that $\phi(a) \subseteq a$ for all $a \in D$.   The functions \[ \prod_{i \in \phi(a)} x_i \prod_{i\in a \setminus \phi(a)}(1+x_i)\] for $a \in D$ form a basis for $\{f:D \to GF(2)\}$.
\end{lemma}
\begin{proof} 
Let $\mathcal{B}$ denote the set of polynomials that we wish to show is a basis.
Since the cardinality of $\mathcal{B}$ is $|D|$ it is enough to show that it is a spanning set.
By Lemma~\ref{basis1}, it is enough to show that every monomial of the form $\prod_{i\in a}x_i$ for $a\in D$ can be expressed as a linear combination of polynomials in $\mathcal{B}$.
We proceed by induction on the size of $a$. The case of $a=\emptyset$ is trivial.
For the induction step, let $a\in D$ be non-empty. Expand the polynomial
\[\prod_{i \in \phi(a)} x_i \prod_{i\in a\setminus\phi(a)}(1+x_i) = \left(\prod_{i \in a} x_i\right) + r,\]
where $r$ is a linear combination of monomials $\prod_{i\in b}x_i$ for $b \subseteq a$ and $b \neq a$.
Since $D$ is downward-closed, by induction hypothesis $r$ is in the span of $\mathcal{B}$. 
Thus, 
$$\prod_{i\in a}x_i = \left(\prod_{i \in \phi(z)} x_i \prod_{i\in a\setminus\phi(a)}(1+x_i)\right) + r$$
is also in the span of $\mathcal{B}$, and we are done.
 \end{proof}

\begin{proof}[Proof of Theorem~\ref{pseudo}] We show there exists a bijection $\pi : D\to D$  such that $a \cap \pi(a) = \phi(a)$ for all $a\in D$, for the given map $\phi$.   Consider the $|D| \times |D|$ boolean matrix $M$ defined as follows.   Index the rows and columns both by $D$, and define the element in location $(a,b) \in D \times D$ to be one if and only if $a \cap b = \phi(a)$.  We claim that $M$ is nonsingular.  Indeed, the rows of $M$ correspond to the functions in Lemma~\ref{basis2}.  Since they form a basis,  the row space of $M$ is $|D|$-dimensional.  This implies the determinant of $M$ is nonzero.  There must exist a permutation  $\pi: [n] \to [n]$ such that \[\prod_{i = 1}^{|D|} M_{i,\pi(i)} = 1.\]   By the definition of $M$, we found the  bijection $\pi$ we were looking for.  
\end{proof}

\section{Conclusion}
We exhibited a connection between algebra and combinatorics.  We provided a general way to lower bound the Hilbert function.  We showed a new family of functions cannot be approximated by low-degree polynomials.  We provided a polynomial method proof of the Sandwich theorem and for a new theorem about prescribed intersections.   

\subsection{Open Directions}
We have introduced the interpolation degree as a new complexity measure on subsets of the boolean hypercube.   
We find it interesting  to understand the structure of sets with low interpolation degree.  The following fact gives a simple characterization of sets with interpolation degree one. 
\begin{proposition}  
A set $C\subseteq \bitsz$ has $\intdeg(C) = 1$ if and only if it is affinely independent.
\end{proposition}
To see why this holds, note that $\intdeg(C)=1$ if and only if the polynomials $1,x_1,\ldots,x_n$ 
span $\{f:C\rightarrow\F\}$. This is equivalent to saying that the matrix
\[	
	\left( \begin{array}{cccc}
	1 & c_{1,1} & \ldots & c_{1,n} \\ 
	. & . & \ldots & . \\
	. & . & \ldots & . \\
	. & . & \ldots & . \\
	1 & c_{m,1} & \ldots & c_{m,n} 
	\end{array} \right )
\]
has rank $m=|C|$. Now, this matrix has rank $m$ if and only if its rows are independent, which exactly means that $C$ consists of $m$ affinely independent vectors.

Is there a simple characterization of sets having interpolation degree two?
We remark that the fact that $\intdeg(C)\leq\vc(V)$ implies that such a characterization will also shed light on the structure of sets with VC dimension two, for which our understanding is lacking~\cite{AlonMY14,MoranSWY15}.

\section*{Acknowledgements}
We thank Paul Beame, Eli Ben-Sasson, Yuval Filmus, Ariel Gabizon, Sivaramakrishnan Natarajan Ramamoorthy, Anup Rao, Mert Sa$\mathrm{\breve{g}}$lam, and Amir Yehudayoff for helping us improve the presentation of the paper.

\newpage
\bibliographystyle{alpha}
\bibliography{sauersbib} 
\newpage
\begin{appendix}

\section{Proof of Smolensky's Lower bound}\label{app:smolensky}

For completeness, we provide a proof of  Smolensky's Theorem~\ref{smolensky}.       We introduce the following notation that will be convenient.   For any set $C \subseteq \{0, 1\}^n$ and degree $d \in [n]$ let $\edc$ denote the evaluation (or restriction) mapping 
\[\edc : \F[x_1,\ldots,x_n]^{\leq d} \to \{f:C \to \F\},\]
where  $\F[x_1,\ldots,x_n]^{\leq d}$ denotes polynomials with degree at most $d$ over $\F$.  Note that $\edc$ is a linear map and  $h_d(C,\F) = \rk_\F(\edc)$.  Associate the map $\edc$ with a matrix such that the rows correspond to elements of $C$, the columns correspond to degree $d$ monomials, and entries are evaluations. 

\subsection{Proof of Theorem~\ref{smolensky}}
Smolensky's Theorem follows from two lemmas.  We assume all ranks are over $\F$ for convenience.  

\begin{lemma}\label{symmetric-difference}For any $S, T \subseteq \{0,1\}^n$ and any $d \in [n]$ we have
$ |S \setminus T| \geq h_d(S,\F) - h_d(T,\F).$
\end{lemma}
\begin{proof}  Notice that $\eds$ and $\edt$ have identical submatrices induced by the $S \cap T$ rows.  Denote the submatrices as $A = \mathsf{E}^{d,S\cap T}$ and $B = \mathsf{E}^{d,S\setminus T}$ and $C = \mathsf{E}^{d,T\setminus S}$.  We want to prove 
\begin{eqnarray}  \rk(B) \geq \rk \left( \begin{array}{c} A  \\ B  \end{array} \right ) -   \rk \left( \begin{array}{c} A \\ C  \end{array} \right ). \label{eqn1}
\end{eqnarray}
 Equation~\ref{eqn1} actually holds for any three matrices, following from the bounds 
\[\rk \left( \begin{array}{c} A  \\ B  \end{array} \right ) \leq \rk(A) +\rk(B) \ \ \ \ \ \ \mathrm{ and } \ \ \ \ \ \ \rk \left( \begin{array}{c} A  \\ C  \end{array} \right ) \geq \rk(A).  \] 
The fact that $\rk(B) = \rk(\mathsf{E}^{d,S\setminus T}) \leq |S \setminus T|$ concludes the proof.
\end{proof}


\begin{lemma}[\cite{smolensky}]\label{lemma2} Let $p:\bitsz \to \F$ and define  $P = \{x  :  p(x) \neq 0\}$.   If $d < (n - \deg(p))/2$ then 
\[ \rk(\mathsf{E}^{d,P}) \leq |P|/2. \] 
\end{lemma}
\begin{proof} Assume for contradiction that $\rk(\mathsf{E}^{d,P}) = r > |P|/2$.  
Let $M_1$ be an $r \times r$ full-rank sub-matrix.  Our goal is to find two degree $d$  polynomials $q_1$ and $q_2$ such that the product $q_1q_2$ is the indicator function for the first row in $M_1$.  
Assume without loss of generality the matrix $\mathsf{E}^{d,P}$ looks like 
\[\left( \begin{array}{c|c} M_1 & \cdots \\\hline M_2 & \cdots \end{array}\right )\]
Let $Q_1$ and $Q_2$ denote the rows corresponding to $M_1$ and $M_2$, respectively.  We start by constructing $q_2$ to be zero on $Q_2$.  Since the rank satisfies $r > |P|/2$ we know $M_2$ has more columns than rows, and thus we can find a vector $v_2$ such that $\mathsf{E}^{d,P}v_2$ is zero on $Q_2$.  Since $M_1$ has full-rank, we know that $\mathsf{E}^{d,P}v_2$ is nonzero on $Q_1$.  Notice that for any $v \in \F^{|P|}$ the vector $\mathsf{E}^{d,P}v$  corresponds to a degree $d$ polynomial.   Let $q_2$ be the polynomial corresponding to $v_2$. 

Since $M_1$ has full rank, we can find a vector $v_1$ and corresponding polynomial $q_1$ such that $q_1q_2$ is the indicator function for the first row of $Q_1$.  To conclude, notice that $pq_1q_2$ interpreted as a function on $\{0,1\}^n$ is nonzero on only a single point.  By inspection this means $\deg(pq_1q_2) = n$, contradicting the assumption that its degree is at most $2d + \deg(p) < n$. 

\end{proof}

With the above lemmas in hand, we can prove Smolensky's theorem.  
\begin{proof}[Proof of Theorem~\ref{smolensky}]  Let $F = f^{-1}(0)$ and $P = \{x  :  p(x) \neq 0\}$ 
 and $d = (n-\deg(p)-1)/2$.  Notice \[\Pr_x[p(x) \neq f(x)] \geq |F\triangle P|/2^{n}.\]  We show $|F\triangle P| \geq 2h_d(F,\F)- |F|.$  Indeed, this follows from the above two lemmas:
\begin{align*}
|F\setminus P| &\geq h_d(F,\F)-h_d(P,\F)\tag{By Lemma~\ref{symmetric-difference}}\\
                        &\geq h_d(F,\F)-\frac{|P|}{2}\tag{Lemma~\ref{lemma2} implies $h_d(P,\F)\leq\frac{|P|}{2}$}\\
                        &= h_d(F,\F)-\frac{|F|+|P\setminus F|-|F\setminus P|}{2}.\tag{$|P| = |F|+|P\setminus F|-|F\setminus P|$.} 
\end{align*}
Therefore, by rearranging the above expressions, we see that 
$$\frac{1}{2}|F \triangle P| \geq h_d(F,\F)- \frac{|F|}{2}.$$
 \end{proof}

\section{Downward-closed Bases from Shifting}\label{app:shifting}
In this section, we prove Theorem~\ref{thm:shifting}, which says that $\{x_I : I\in \setof{D}\}$
is a basis for the vector space of functions $\{f:C\rightarrow \F\}$,  where $D=S_{n}(S_{n-1}(\ldots S_{1}(C)))$. The proof of Theorem~\ref{thm:shifting} is by induction. For the induction to work we prove a stronger statement. 
Consider the lexicographical order on the set of all multilinear monomials.
That is, $m_1 < m_2$ if the smallest $i$ such that $x_i$ appears in exactly one of $m_1,m_2$ appears in $m_2$. 
For example,
$x_1 > x_2x_3x_4\ldots x_{n}$.

We are now ready to state the stronger statement we will prove.
\begin{lemma}\label{lem:leading_monomials}
Let $p$ be a polynomial where $x_I$ is the leading monomial of $p$. Then
 $$p|_C=0 \implies I\notin \setof{D}.$$
\end{lemma}
\noindent
The above lemma implies
that there is no linear combination of monomials in $\{x_I:I\in \setof{D}\}$
which represents the zero function on $C$ (because the leading monomial in such a combination will be in $\setof{D}$ which contradicts the lemma). Since down-shifting $C$ does not change its cardinality, it follows that $|C|=|\{x_I:I\in\setof{D}\}|$ and therefore $\{x_I:I\in \setof{D}\}$ is indeed a basis. 

In the proof we will use the following ``locality'' of down-shifts. 
For $i \in [n]$ define the two sets 
\[ C_{i=0} = \{c\in C : c_i=0\}, \ \ \ \ \ \  \mathrm{ and } \ \ \ \ \ \  C_{i=1} = \{c\in C : c_i=1\}.\]
For all $i,j \in [n]$ with $i\neq j$, the shifting operator $S_j$ satisfies  
\begin{equation}~\label{eq:local}
S_j(C) = S_j(C_{i=0})\cup S_j(C_{i=1}).
\end{equation}
Putting it differently: $C_{i=0},C_{i=1}$ are invariant under $S_j$ when $i\neq j$.
\begin{proof}[Proof of Lemma~\ref{lem:leading_monomials}]
Let $C_0 = \{c\in C : c_n=0\}$, and $C_1 = \{c\in C : c_n=1\}$.
By the locality of down-shifts (Equation~\ref{eq:local}):
$$S_{n-1}(S_{n-2}(\ldots S_1(C))) = S_{n-1}(S_{n-2}(\ldots S_2(C_0)))\cup S_{n-1}(S_{n-2}(\ldots S_2(C_1))).$$
For $b\in\{0,1\}$, let \[D_b=\{w|_{\{1,2,\ldots,n-1\}}:w\in S_{n-1}(S_{n-2}(\ldots S_1(C_b)))\}.\]
By the induction hypothesis applied on $C_b|_{\{1,\ldots,n-1\}}$, 
there are no polynomials $p(x_1,\ldots,x_{n-1})$ that represent the zero function on $C_b$ whose leading monomial is $x_I$ where $I\in \setof{D_{b}}$.

Assume towards contradiction some multilinear polynomial $p(x_1,\ldots,x_n)$ represents the zero function on $C$ and has leading monomial of $x_I$ with $I\in D$. 
A crucial observation which follows directly from the definition of the down-shift $S_n$ is that if $n\in I$ then $(I\setminus\{n\})$ belongs to $\setof{D_0}$ {\bf and} to $\setof{D_1}$, and if $n\notin I$ then $I$ belongs to $\setof{D_0}$ {\bf or} to $\setof{D_1}$.
Consider the expansion 
$$p=\alpha_I x_I + \sum_{J < I}{\alpha_J x_J},$$
where $J < I$ in the lexicographical monomial ordering.  We distinguish between two cases.

(i) $n\notin I$. Assume that $I\in D_1$ (the proof for $I \in D_0$ is symmetric). Notice that since $x_I$ is the leading monomial, the monomial $x_{I\cup\{n\}}$ does not appear in $p$. Therefore, the polynomial that is obtained by setting $x_n=1$ in $p$ has leading monomial $x_I$ and represents the zero function on $C_1$, contradicting the induction hypothesis.

(ii) $n\in I$. Here we distinguish between two subcases. If $\alpha_{I\setminus\{n\}}+\alpha_I\neq 0$ then the polynomial that is obtained by setting $x_n=1$ in $p$ has leading monomial $x_{I\setminus\{n\}}$ and represents the zero function on $C_1$, contradicting the induction hypothesis.
If $\alpha_{I\setminus\{n\}}+\alpha_I = 0$ then the polynomial that is obtained by setting $x_n=0$ in $p$ has leading monomial $x_{I\setminus\{n\}}$ and represents the zero function on $C_0$, contradicting the induction hypothesis.
\end{proof}

\section{Hyperplane Arrangement Lemmas}\label{app:hyplerplanes}
We now prove the two lemmas that we assumed in the proof of Proposition~\ref{thm:hyp-extremal}.  Recall $\mathcal{H}$ is in general position if any $k\leq d$ hyperplanes intersect in a $(d-k)$-dimensional affine subspace. This means (i) every $d+1$ hyperplanes have an empty intersection, and (ii) every subset of $d$ normal vectors are linearly independent. (if some $d$ normal vectors are linearly dependent then their intersection  is either empty or infinite but not $0$ dimensional).

\begin{lemma}
Let $\mathcal{H}$ be an arrangement in $\R^d$.
If $Y\subseteq [m]$ is shattered by $S_\mathcal{H}$
then the hyperplanes $h_i$ for $i\in Y$ have a non-empty intersection.
In particular, if $\mathcal{H}$ is in general position then $|Y|\leq d$.
\end{lemma}
\begin{proof}
Let $M_Y$ be the matrix whose rows are the normals $\vec n_i,i\in Y$,
and let $\vec b_Y$ be the vector $(b_i)_{i\in Y}$.
Consider the affine transformation $T_Y:\R^d\rightarrow\R^{Y}$ defined by
$$T_Y(\vec x) = M_Y\cdot \vec x - \vec b_Y.$$
We need to show that there is some $\vec x\in \cap_{i\in Y}h_i$.
Note that $\vec x\in \cap_{i\in Y}$ if and only if $T_Y(\vec x) = \vec 0$.
Thus, it is enough to show that $\vec 0$ is in the image of $T_Y$.
Now, the fact that $S_\mathcal{H}$ shatters $Y$ amounts to that
for every sign vector $\vec s\in\{-1,1\}^Y$ there exists $\vec x\in\R^m$
such that $\sign(T_Y(\vec x)) = \vec s$. In other words, every orthant contains a vector $u$ in the image of $T_Y$. Now, the convex hull, $C$, of these vectors is also contained in the image
of $T_Y$ (since the image of an affine transformation is convex). Also, since these vectors contain a vector in every orthant, it follows  that $\vec 0\in C$, and therefore $\vec 0$ is in the image of $T_Y$.  Indeed, an exercise shows that any set of vectors containing a vector in each orthant also contains the origin in its convex hull.
\end{proof}
\begin{lemma}
Let $\mathcal{H}$ be an arrangement in $\R^d$ that is in general position.
If the normals $\vec n_i, i\in Y$ are linearly independent
then $Y$ is strongly shattered by $S_\mathcal{H}$.
In particular, any subset of size at most $d$ is strongly shattered by $S_\mathcal{H}$.
\end{lemma} 
\begin{proof}
It is enough to prove only for sets $Y$ such that the normals $\vec n_i, i\in Y$ form a basis to $\R^d$ (because every independent set can be extended to a basis and every subset of a strongly shattered set is strongly shattered).
Let $M=M_{[m]}$ be the matrix whose rows are the normals $\vec n_i,i\in [m]$,
and let $\vec b = \vec b_{[m]}$ be the vector $(b_i)_{i\in [m]}$.
Consider the affine transformation $T:\R^d\rightarrow\R^{m}$ defined by
$$T(\vec x) = M\cdot \vec x - \vec b.$$
Since $\mathcal{H}$ is in general position, there is a unique $\vec x^*$ that lies in the intersection $\cap_{i\in Y}h_i$. In other words, $T(\vec x^*)_i = 0$ for every $i\in Y$.
Moreover, since every $d+1$ hyperplanes have an empty intersection,
it follows that for all $j\notin Y$: $T(\vec x^*)_j\neq 0$.
We show that $Y$ is strongly shattered by showing that for every sign pattern
$s\in\{- 1,1\}^Y$ there exists $\vec u\in\R^d$ such that for all $i\in Y$
$$\sign(T(\vec x^*+\vec u))_i = \vec s_i,$$
and for all $j\notin Y$
$$\sign(T(\vec x^*+\vec u))_j = \sign(\vec x^*)_j.$$
Indeed, since the normals $\vec n_i, i\in Y$ are linearly independent,
there is some $\vec w\in\R^d$ such that $(M\cdot\vec w)_i = s_i$.
Therefore, for every $\eps>0$ and $i\in Y$
$$\sign(T(\vec x^*+\eps\vec w))_i = s_i.$$
Moreover, since for all $j\notin Y$ $T(\vec x^*)_j \neq 0$
then the exists some $\eps'>0$ such that for al $j\notin Y$
$$\sign(T(\vec x^*+\eps'\vec w))_j = \sign(\vec x^*)_j.$$
Thus, picking $\vec u=\eps' \vec w$ finishes the proof.
\end{proof}

\end{appendix}
\end{document}